\documentclass[a4paper,11pt]{article}
\usepackage[latin1]{inputenc}
\usepackage{amsfonts}
\usepackage{amsmath}
\usepackage{latexsym}
\usepackage[all]{xy}
\usepackage{amsthm}
\usepackage{graphicx}
\usepackage[T1]{fontenc}

\usepackage{hyperref}

\usepackage{caption2}
\usepackage{hyperref}
\usepackage{float}
\usepackage{amssymb}
\usepackage{amscd}
\author{Álvaro Martínez-Pérez\footnote{Partially supported by MTM 2006-00825}}

\date{Departamento de Geometría y Topología, Universidad Complutense de Madrid. Madrid 28040, Spain\\
        e-mail: alvaro\_martinez@mat.ucm.es}

\title{Quasi-isometries between visual hyperbolic spaces.}

\begin{document}
\maketitle

\newtheorem{definicion}{Definition}[section]
\newtheorem{nota}[definicion]{Remark}
\newtheorem{prop}[definicion]{Proposition}
\newtheorem{lema}[definicion]{Lemma}
\newtheorem{obs}[definicion]{Remark}
\newtheorem{teorema}[definicion]{Theorem}
\newtheorem{cor}[definicion]{Corollary}
\newtheorem{ejp}[definicion]{Example}

\begin{abstract} We prove that a PQ-symmetric homeomorphism
between two complete metric spaces can be extended to a
quasi-isometry between their hyperbolic approximations.

This result is used to prove that two visual Gromov hyperbolic
spaces are quasi-isometric if and only if there is a PQ-symmetric
homeomorphism between their boundaries.
\end{abstract}

\section{Introduction.}

There are many results studying the geometry of hyperbolic spaces
from a large scale point of view by looking at the boundary. Many
of them are motivated by questions about Gromov hyperbolic groups
and the results involve group actions or other techniques out of
the geometric framework. For example, F. Paulin (see \cite{Pau})
characterizes, from the boundary, the quasi-isometries between
Gromov hyperbolic spaces under the assumption that there is a
group acting isometrically and co-compactly on them. There are
also many further results involving group actions, considering
quasi-conformal structures on the boundary.

Other works, like \cite{BoS} and \cite{BS}, which is the main
source for this paper, restrict themselves to Gromov hyperbolic
spaces as geometric objects. In \cite{BoS} appears some useful
construction, the hyperbolic cone $X$ over a bounded metric space
$Z$, which is a hyperbolic space whose boundary is identified with
$Z$, $\partial_\infty X =Z$, and where the original metric in $Z$
is a visual metric for $\partial_\infty X$. Then, they prove that
PQ-symmetric maps between bounded metric spaces can be extended to
quasi-isometries between their hyperbolic cones.

In \cite{BS}, S. Buyalo and V. Schroeder introduce a special kind
of hyperbolic cones called hyperbolic approximations, which are
defined in general for non-necessarily bounded metric spaces. This
hyperbolic approximation has the advantage of being geodesic
(while the hyperbolic cone is only roughly geodesic) and also, of
including in the construction fixed coverings by balls of the
metric space. With this, they obtain some extension of the
mentioned result in of M. Bonk and O. Schramm proving that
quasi-symmetric homeomorphisms between uniformly perfect, complete
metric spaces can be extended to quasi-isometries between the
hyperbolic approximations and they characterize from the boundary
the quasi-isometry type of visual hyperbolic spaces with uniformly
perfect boundary.

Herein, we generalize this result for complete metric spaces and
therefore, give a necessary and sufficient condition on the map
between the boundaries for two visual hyperbolic spaces to be
quasi-isometric.

The main results would be the following.

\begin{teorema}\label{tma} For any PQ-symmetric homeomorphism $f:Z\to Z'$ of complete
metric spaces, there is a quasi-isometry of their hyperbolic
approximations $F: X \to X'$ which induces $f$, $\partial_\infty
F(z)=f(z) \, \forall z\in Z$. Moreover, this quasi-isometry can be
made continuous.
\end{teorema}

\begin{cor} Let $X,X'$ be visual hyperbolic geodesic spaces. Then,
any PQ-symmetric homeomorphism $f:\partial_\infty X \to
\partial_\infty X'$ can be extended to a quasi-isometry $F:X \to
X'$.
\end{cor}

\begin{cor} Two visual hyperbolic geodesic spaces $X,\ Y$ are quasi-isometric
if and only if there is a PQ-symmetric homeomorphism $f$ with
respect to any visual metrics between their boundaries with base
points in $X$, $Y$ respectively.
\end{cor}

\textbf{Acknowledgments.} The author would like to express his
gratitude to Bruce Hughes, for his help and support producing this
work.

\section{The boundary at infinity of a Gromov hyperbolic space.}

We recall some basic definitions about Gromov hyperbolic spaces.
There are many references where a more detailed and deeper
exposition can be found. Let us cite among them just the work of
Gromov \cite{Gr}, and the well known book of Ghys and de la Harpe,
\cite{G-H}.

Let $X$ be a metric space. Fix a base point $o\in X$ and for
$x,x'\in X$ put $(x|x')_o=\frac{1}{2}(|xo|+|x'o|-|xx'|)$ where
$|xy|$ denotes the distance between $x,y$. The number $(x|x')_o$
is non-negative and it is called the \emph{Gromov product} of
$x,x'$ with respect to $o$.

\begin{definicion} A metric space $X$ is \emph{(Gromov)
hyperbolic} if it satisfies the $\delta$-inequality
\[(x|y)_o\geq min\{(x|z)_o,(z|y)_o\}-\delta\] for some $\delta\geq
0$, for every base point $o\in X$ and all $x,y,z \in X$.
\end{definicion}

Let $X$ be a hyperbolic space and $o\in X$ a base point. A
sequence of points $\{x_i\}\subset X$ \emph{converges to infinity}
if \[\lim_{i,j\to \infty} (x_i|x_j)_o=\infty.\] This property is
independent of the choice of $o$ since
\[|(x|x')_o-(x|x')_{o'}|\leq |oo'|\] for any $x,x',o,o' \in X$.

Two sequences $\{x_i\},\{x'_i\}$ that converge to infinity are
\emph{equivalent} if \[\lim_{i\to \infty} (x_i|x'_i)_o=\infty.\]
Using the $\delta$-inequality, we easily see that this defines an
equivalence relation for sequences in $X$ converging to infinity.
The \emph{boundary at infinity} $\partial_\infty X$ of $X$ is
defined to be the set of equivalence classes of sequences
converging to infinity.

The notion of Gromov product can be extended to points in the
boundary. Let $\xi,\xi'\in \partial_\infty X$. Define their Gromov
product by
\begin{equation} \label{gr} (\xi|\xi')_o=\inf \underset{i\to \infty}{\lim
\inf}(x_i|x'_i)_o \end{equation} where the infimum is taken over
all sequences $\{x_i\}\in \xi$ and $\{x'_i\}\in \xi'$.

A metric $d$ in $\partial_\infty X$ is said to be \emph{visual} if
there are $o\in X$, $a>1$ and positive constants $c_1,c_2$ such
that \[c_1a^{-(\xi|\xi')_o} \leq d(\xi | \xi')_o \leq
c_2a^{-(\xi|\xi')_o}\] for all $\xi,\xi'\in \partial_\infty X$. In
this case, we say that $d$ is a \emph{visual metric} with respect
to the base point $o$ and the parameter $a$.

\begin{teorema} Let $X$ be a hyperbolic space. Then for any $o\in
X$, there is $a_0>1$ such that for every $a\in (1,a_0]$ there
exists a metric $d$ on $\partial_\infty X$ which is visual with
respect to $o$ and $a$.
\end{teorema}

In case we have $x\in X$ and $\xi\in \partial_\infty X$, to define
$(\xi,x)_o$ consider in (\ref{gr}) $x'_i=x$.

\begin{definicion} A hyperbolic space is \emph{visual} if for some
base point $o\in Y$ there is a positive constant $D$ such that for
every $y\in Y$ there is $\xi\in \partial_\infty Y$ with
$d(o,y)\leq (y|\xi)_o+D$.
\end{definicion}

\section{Hyperbolic approximation of metric spaces.}

We recall here the construction of the hyperbolic approximation
introduced in \cite{BS}.

A subset $V$ in a metric space $Z$ is called \emph{a-separated},
$a>0$, if $d(v,v')\geq a$ for any distinct $v,v'\in V$. Note that
if $V$ is maximal with this property, then the union $\cup_{v\in
V} B_a(v)$ covers $Z$.

A \emph{hyperbolic approximation} of a metric space $Z$ is a graph
$X$ which is defined as follows. Fix a positive $r\leq
\frac{1}{6}$ which is called the \emph{parameter} of $X$. For
every $k\in \mathbb{Z}$, let $V_k\in Z$ be a maximal
$r^k$-separated set. For every $v\in V_k$, consider the ball
$B(v)\subset Z$ of radius $r(v):=2r^{k}$ centered at $v$. Let us
fix more precisely the set $V$ as the union, for $k\in
\mathbb{Z}$, of the set of balls $B(v)$, $v\in V_k$. Therefore, if
for any $v,v'\in V_k$, $B(v)=B(v')$, they represent the same point
in $V$, but if $B(v_k)=B(v_{k'})$ with $k\neq k'$, then they yield
different points in $V$. Let $V$ be the vertex set of a graph $X$.
Vertices $v,v'$ are connected by an edge if and only if they
either belong to the same level, $V_k$, and the close balls
$\bar{B}(v),\bar{B}(v')$ intersect, $\bar{B}(v)\cap
\bar{B}(v')\neq \emptyset$, or they lie on neighboring levels
$V_k,V_k+1$ and the ball of the upper level, $V_{k+1}$, is
contained in the ball of the lower level, $V_k$.

An edge $vv'\subset X$ is called \emph{horizontal}, if its
vertices belong to the same level, $v,v'\in V_k$ for some $k\in
\mathbb{Z}$. Other edges are called \emph{radial}. Consider the
path metric on $X$ for which every edge has length 1. $|vv'|$
denotes the distance between points $v,v'\in V$ in $X$, while
$d(v,v')$ denotes the distance between them in $Z$. There is a
natural level function $l:V \to \mathbb{Z}$ defined by $l(v)=k$
for $v\in V_k$. Consider also the canonical extension $l:X \to
\mathbb{R}$.

Note that any (finite or infinite) sequence $\{v_k\}\in V$ such
that $v_kv_{k+1}$ is a radial edge for every $k$ and the level
function $l$ is monotone along $\{v_k\}$, is the vertex sequence
of a geodesic in $X$. Such a geodesic is called \emph{radial}.

Assume now that the metric space $Z$ is bounded and non-trivial.
Then the largest integer $k$ with $diam Z<r^k$ exists, and it is
denoted by $k_0=k_0(diam Z,r)$. For every $k\leq k_0$ the vertex
set $V_k$ consists of one point, and therefore contains no
essential information about $Z$. Thus, the graph $X$ is modified
putting $V_k=\emptyset$ for every $k<k_0$ and this modified graph
is called the \emph{truncated hyperbolic approximation} of $Z$.

\vspace{0.5cm}

Proposition 6.2.10 in \cite{BS} states that:

\begin{prop} A hyperbolic approximation of any metric space is a
geodesic $2\delta$-hyperbolic space with $2\delta=3$.
\end{prop}

Also, if $X$ is the hyperbolic approximation of a complete metric
space $Z$, there is a canonical identification $\partial_\infty
X=Z\cup \{\infty \}$ such that the metric of $Z$ is visual on
$\partial_\infty X \backslash \{\omega\}$, where $\{\omega\}$ is
the unique point at infinity represented by a sequence $\{v_i\}\in
V$ with $l(v_i)\to -\infty$ and corresponds to the point
$\{\infty\}$ added to $Z$. If $Z$ is bounded and $X$ is the
truncated hyperbolic approximation of $Z$ then $\partial_\infty
X=Z$.

\vspace{0.5cm}

Let us recall the following lemma in \cite{BS} (6.2.2).

\begin{lema} \label{branch} For every $v,v'\in V$ there exists $w\in
V$ with $l(w)\leq l(v),l(v')$ such that $v,v'$ can be connected to
$w$ by radial geodesics.
\end{lema}

Consider any subset $V'\subset V$. A point $u\in V$ is a
\emph{cone} point for $V'$ if $l(u)\leq inf_{v\in V'} l(v)$ and
every $v\in V'$ is connected to $u$ by a radial geodesic. A cone
point of maximal level is called a \emph{branch point} of $V'$. By
lemma \ref{branch}, for any two points $v,v'\in V$ there is a cone
point. Thus every finite $V'$ possesses a cone point and hence a
branch point.

\begin{definicion} $v_k\in V_k$ is a splitting point of $V$ if
there is some $v_{k+1}\in V_{k+1}$ such that $B(v_{k+1})\subsetneq
B(v_k)$.
\end{definicion}

\begin{lema}\label{split} Any branch point for two vertices which are not joined by a radial
geodesic is a splitting point.
\end{lema}

\begin{proof} Suppose $v_k$ is a cone point of $v,v'$ and it is not
a splitting point. For any $v_{k+1}\in V_{k+1}$ with
$B(v_{k+1})\subset B(v_k)$ by definition $B(v_{k+1})=B(v_k)$.
Then, since $v_k$ is joined to $v$ and $v'$ by radial geodesics,
$v_{k+1}$ is also a cone point and then, the branch point is at
least in level $k+1$.
\end{proof}

\begin{lema}\label{diam1} Suppose $v_k$ is a splitting point in a hyperbolic
approximation with parameter $r$. Then,
\[r^{k+1}\leq diam B(v_k) \leq 4r^{k}.\]
\end{lema}

The following lemmas appear as 6.2.5, 6.2.6 and 6.2.7 in
\cite{BS}.

\begin{lema}\label{geod1} Any two vertices $v,v'\in V$ can be joined by a
geodesic $\gamma=v_0,...,v_{n+1}$ such that
$l(v_i)<max\{l(v_{i-1}),l(v_{i+1})\}$.
\end{lema}

\begin{lema} Any vertices $v,v'\in V$ can be connected in $X$ by a
geodesic which contains at most one horizontal edge. If there is
such an edge, then it lies on the lowest level of the geodesic.
\end{lema}

\begin{lema}\label{intersect} Assume that for some $v,v'\in V$ the balls $B(v),\
B(v')$ intersect. Then $|vv'|\leq |l(v)-l(v')|+1$.
\end{lema}

The same argument of \ref{intersect} yields:



\begin{lema}\label{cone} Assume that for some $v,v'\in V$ there is some vertex $v''$ with
$l(v'')\leq l(v),l(v')$ and such that $B(v'')$ intersects the
balls $B(v),\ B(v')$. Then, there exists a cone point $w$ for
$v,v',v''$ such that $l(w)=l(v'')-1$.
\end{lema}

\begin{proof} Suppose $B(v'')=B(v'',2r^k)$. There is some point $w''\in
B(v'',r^{k-1})\cap V_{k-1}$ and clearly $B(v'',2r^k)\subset
B(w,2r^{k-1})$. Also, there exist $u,u'\in V_k$ joined by radial
geodesics with $v,v'$, and since $B(v'')$ intersects $B(v),\
B(v')$, it intersects in particular $B(u), \ B(u')$. Hence
$|uv''|,|u'v''|\leq 1$ and $B(u),B(u')$ are contained in
$B(v'',6r^{k})\subset B(w,2r^{k-1})$.
\end{proof}

\begin{lema}\label{diam2} Let $v_1,v_2$ be any two vertices
in a hyperbolic approximation with parameter $r$. If $w$ is a
branch point for $v_1,v_2$, then
\[\frac{r^2}{4}\leq \frac{diam(B(v_1)\cup B(v_2))}{diam (B(w))}.\]
\end{lema}

\begin{proof} By Lemma \ref{cone}, there is no vertex at level
$l(w)+2$ whose ball contains both $B(v_1),\ B(v_2)$ and hence
\[diam(B(v_1)\cup B(v_2))\geq \frac{r^2}{4} diam (B(w))).\]
\end{proof}

\section{Extension of quasi-isometries.}

The following definitions are classical in asymptotic geometry.
However, there are different conventions in the literature for
some of them. (In \cite{BoS}, for example, they use the term
"roughly quasi-isometric" instead of "quasi-isometric" keeping
this name for the particular case when the additive constant is
0). Let us fix this concepts as they are stated in \cite{BS} which
is the main reference for this paper.

\begin{definicion} A subset $A\subset Y$ in a metric space $Y$ is
called a \emph{net} if there is a constant $D>0$ such that
$\forall y\in Y$, $d(y,A)\leq D$.
\end{definicion}

\begin{definicion} A map between metric spaces, $f:(X,d_X)\to (Y,d_Y)$, is
\emph{rough isometric} if there is a constant $C>0$ such that
$\forall x,x'\in X$, $|d_{Y}(f(x),f(x'))-d_X(x,x')|\leq C$. If
$f(X)$ is a net in $Y$, then $f$ is a \emph{rough isometry} and
$X,Y$ are \emph{roughly isometric}.
\end{definicion}

\begin{definicion} A map between metric spaces, $f:(X,d_X)\to (Y,d_Y)$, is
said to be \emph{homothetic} if if there are constants $a,b$ such
that $\forall x,x'\in X$, $|d_Y(f(x),f(x'))-a\cdot d_X(x,x')|\leq
b$. If $f(X)$ is a net in $Y$, then $f$ is a \emph{rough
similarity} and $X,Y$ are \emph{roughly similar}.
\end{definicion}

\begin{definicion} A map between metric spaces, $f:(X,d_X)\to (Y,d_Y)$, is
said to be \emph{quasi-isometric} if there are constants $\lambda
\geq 1$ and $C>0$ such that $\forall x,x'\in X$,
$\frac{1}{\lambda}d_X(x,x') -A \leq d_Y(f(x),f(x'))\leq \lambda
d_X(x,x')+A$. If $f(X)$ is a net in $Y$, then $f$ is a
\emph{quasi-isometry} and $X,Y$ are \emph{quasi-isometric}.
\end{definicion}




See \cite{BoS} and \cite{BS}. A quasi-symmetric
homeomorphism in the boundaries can be extended to a
quasi-isometry for visual hyperbolic spaces with uniformly perfect
boundaries at infinity. See Theorem 7.2.1 an Corollary 7.2.3 in
\cite{BS}.

\begin{definicion} A map $f:X \to Y$ between metric spaces is
called \emph{quasi-symmetric} if it is not constant and if there
is a homeomorphism $\eta:[0,\infty) \to [0,\infty)$ such that from
$|xa|\leq t|xb|$ it follows that $|f(x)f(a)|\leq
\eta(t)|f(x)f(b)|$ for any $a,b,x\in X$ and all $t\geq 0$. 
The function $\eta$ is called the \emph{control function} of $f$.
\end{definicion}

\begin{definicion} A quasi-symmetric map is said to be \emph{power quasi-symmetric}
or \emph{PQ-symmetric}, if its control function is of the form
\[\eta(t)= q \max\{t^p,t^{\frac{1}{p}}\}\] for some $p,q\geq 1$.
\end{definicion}

\begin{definicion} A map between metric spaces is said to be \emph{bounded} if the
image of any bounded set is bounded.
\end{definicion}

\begin{prop}\label{pq} A map $f$ between metric spaces is PQ-symmetric if and only
if it is bounded and there exist constants $\lambda\geq 1, A>0$
such that for any pair of non-trivial (i.e. with at least two
points) bounded sets $B_2 \subset B_1$,

\begin{equation}\label{quasiprop} A\Big(\frac{diam (B_2)}{diam (B_1)}\Big)^\lambda \leq
\frac{diam (f(B_2))}{diam (f(B_1))} \leq
\frac{1}{A}\Big(\frac{diam (B_2)}{diam
(B_1)}\Big)^\frac{1}{\lambda}.\end{equation}
\end{prop}

\begin{proof} Suppose $f$ is a PQ-symmetric map with constants $p,q$. Obviously
it is bounded.  Let $B_2\subset B_1$ be any pair of non-trivial
bounded sets and $x,a$ any pair of points in $B_2$ with $|xa|\geq
\frac{1}{2}diam(B_2)$. Consider $t\geq 1$ such that $|xb|\leq
t\cdot |xa|$ for every $b\in B_1$ and such that there exists
$b_0\in B_1$ with $|xb_0|\geq \frac{t}{2}|xa|$.

Clearly, \begin{equation}\label{lema1} \frac{diam (B_2)}{diam
(B_1)}\leq \frac{|xa|}{\frac{t}{2}|xa|}=\frac{4}{t}.\end{equation}

Since $f$ is PQ-symmetric with constants $p,q$, $|x'b'|\leq
qt^p|x'a'|$ for $x'=f(x),a'=f(a),b'=f(b)$. Thus, \[\frac{diam
(f(B_2))}{diam (f(B_1))}\geq \frac{|x'a'|}{2qt^p |x'a'|}\geq
\frac{1}{2q}\Big(\frac{1}{t}\Big)^p.\]

From this, together with (\ref{lema1}), we obtain that
\begin{equation}\label{cotainf} \frac{diam (f(B_2))}{diam (f(B_1))}\geq \frac{1}{2q}\cdot
\frac{1}{4^p}\cdot \Big(\frac{diam (B_2)}{diam
(B_1)}\Big)^p.\end{equation}

We use a similar argument for the upper bound. Consider $x\in
B_2$, $a\in B_1$ with $|xa|\geq \frac{1}{3}diam(B_1)$ and $t\leq
1$ such that $|xb|<t|xa|$ for every $b$ in $B_2$ and such that
there exist $b_0\in B_2$ with $2|xb|>t|xa|$.

Clearly, \begin{equation}\label{lema2} \frac{diam (B_2)}{diam
(B_1)}\geq
\frac{\frac{t}{2}|xa|}{3|xa|}=\frac{t}{6}.\end{equation}

Since $f$ is PQ-symmetric with constants $p,q$, $|x'b'|\leq
qt^{\frac{1}{p}}|x'a'|$ with $x',a',b'$ denoting $f(x),f(a),f(b)$.
Thus,
\[\frac{diam (f(B_2))}{diam (f(B_1))}\leq \frac{2qt^{\frac{1}{p}}|x'a'|}
{|x'a'|}\leq 2qt^\frac{1}{p}.\]

This, together with (\ref{lema2}) yields
\begin{equation}\label{cotasup} \frac{diam (f(B_2))}{diam (f(B_1))}\leq 2q6^\frac{1}{p}
\cdot \Big(\frac{diam (B_2)}{diam
(B_1)}\Big)^\frac{1}{p}.\end{equation}

Therefore, it suffices to consider $\lambda=p$ and
$A:=1/\max\{2q4^p,2q6^{\frac{1}{p}}\}$.


Now let $a,b,x\in X$ with $|xb|\leq t|xa|$. Define $B_2:=\{a,x\}$
and $B_1:=\{a,b,x\}$. Clearly, $t|xa| \leq diam (B_1) \leq (t+1)
|xa|$ and \[\frac{1}{t+1}\leq \frac{diam (B_2)}{diam (B_1)}\leq
\frac{1}{t}.\]

Since $f$ is bounded, the diameter $diam (f(B_i))$ is a positive
real number, and there are constants $\lambda\geq 1,A>0$ such that
\[A\Big(\frac{1}{t+1}\Big)^\lambda \leq \frac{diam (f(B_2))}{diam
(f(B_1))}\leq \frac{|f(x)f(a)|}{|f(x)f(b)|}\] and
\begin{equation}\label{caso1} |f(x)f(b)|\leq diam (f(B_1)) \leq \frac{1}{A}(t+1)^\lambda
\cdot |f(x)f(a)|.\end{equation}

Now let $B'_2:=\{x,b\}$. Clearly, $\frac{1}{t}|xb|\leq
diam(B_1)\leq (1+\frac{1}{t})|xb|$ and \[\frac{t+1}{t} \leq
\frac{diam (B'_2)}{diam (B_1)}\leq t.\]

From (\ref{quasiprop}), we get that
\[\frac{|f(x)f(b)|}{diam(f(B_1))}\leq \frac{1}{A}(t)^{\frac{1}{\lambda}}\]
and \[|f(x)f(b)|\leq \frac{1}{A}(t)^{\frac{1}{\lambda}}\cdot
diam(f(B_1)).\]

This, together with (\ref{caso1}), yields
\begin{equation}\label{caso2} |f(x)f(b)|\leq
\frac{1}{A}(t)^{\frac{1}{\lambda}}\cdot \frac{1}{A}(t+1)^\lambda
\cdot |f(x)f(a)|.\end{equation}

If $t\geq 1$, then $(t+1)^\lambda \leq 2^\lambda \cdot t^\lambda$
and from (\ref{caso1}) we obtain that \[|f(x)f(b)|\leq
\frac{2^\lambda}{A}(t)^\lambda \cdot |f(x)f(a)|.\]

If $t<1$, then $(t+1)^\lambda\leq 2^\lambda$ and from
(\ref{caso2}) we obtain that \[|f(x)f(b)|\leq
\frac{2^\lambda}{A^2} t^ {\frac{1}{\lambda}}\cdot |f(x)f(a)|.\]

Therefore, making $p=\lambda$ and
$q=\max\{\frac{2^\lambda}{A},\frac{2^\lambda}{A^2}\}$, $f$ is
PQ-symmetric.
\end{proof}

\begin{definicion} A map is \emph{metrically proper} if the
inverse image of a bounded set is bounded.
\end{definicion}

\begin{prop}\label{mp} If $f$ is a PQ-symmetric, then it is metrically proper.
\end{prop}

\begin{proof} Let $V$ be a bounded non-trivial set in $Y$ and
suppose $f^{-1}(V)$ is not bounded. Consider $B'_2:=\{y_1,y_2\}$
any pair of points in $V$, $B_2=\{x_1,x_2\}$ with $
x_i=f^{-1}(y_i) \ i=1,2$ and suppose $D=diam(V)$. Now we can
choose $B_2\subset B_1\subset f^{-1}(V)$ with $diam(B_1)$ as big
as we want, and by \ref{pq}, this leads to contradiction with
\[\frac{diam (f(B_2))}{diam (f(B_1))} \leq
\frac{1}{A}\Big(\frac{diam (B_2)}{diam
(B_1)}\Big)^\frac{1}{\lambda}\] for fixed constants $\lambda,A$.
\end{proof}





The following technical lemma will be used in the proof of the
theorem. The inequality is only needed to depend on the fixed
constant, so it is not pretended to be optimal.

\begin{lema}\label{union} Let $A_1\subset D_1, \ A_2 \subset D_2$ be bounded
sets in a metric space with $diam (D_i) \leq a\cdot diam (A_i)$,
$i=1,2$ for some constant $a >1$. Then, $diam (D_1 \cup D_2) <
(4a+2) diam (A_1\cup A_2)$.
\end{lema}

\begin{proof} Consider $y_1,y_2 \in D_1 \cup D_2$ such that
$\frac{1}{2}diam(D_1\cup D_2) < d(y_1,y_2)$. If $y_1,y_2 \in D_i$,
for $i=0,1$ then $diam(D_1\cup D_2)<2 diam (D_i)\leq 2a\cdot
diam(A_i)\leq 2a \cdot diam(A_1\cup A_2)$ holding the condition of
the lemma. Otherwise, suppose (relabelling if necessary) $y_1\in
D_1$ and $y_2\in D_2$. Then, for any $x_1\in A_1$, $x_2\in A_2$,
$\frac{1}{2}diam(D_1\cup D_2) < d(y_1,y_2)\leq d(y_1,x_1)
+d(x_1,x_2)+d(x_2,y_2)\leq a \cdot diam (A_1) +diam (A_1\cup A_2)+
a \cdot diam (A_2)\leq (2a+1)diam(A_1\cup A_2)$ finishing the
proof.
\end{proof}

The main theorem is a natural extension of Theorem 7.4 in
\cite{BoS} (for clarity, the statement is adapted to the
definitions as used along this paper) and Theorem 7.2.1 in
\cite{BS}:

\begin{teorema}[Bonk-Schramm] Any PQ-symmetric homeomorphism
$f:Z\to Z'$ of bounded metric spaces can be extended to a
quasi-isometry between their hyperbolic cones $\hat{f}: Con(Z) \to
Con(Z')$.
\end{teorema}

\begin{teorema}[Buyalo-Schroeder] For any quasi-symmetric homeomorphism
$f:Z\to Z'$ of uniformly perfect, complete metric spaces, there is
a quasi-isometry of their hyperbolic approximations $F: X \to X'$
which induces $f$, $\partial_\infty F(z)=f(z) \; \forall z\in Z$.
\end{teorema}

\begin{teorema}\label{tma} For any PQ-symmetric homeomorphism $f:Z\to Z'$ of complete
metric spaces, there is a quasi-isometry of their hyperbolic
approximations $F: X \to X'$ which induces $f$, $\partial_\infty
F(z)=f(z) \; \forall z\in Z$. Moreover, this quasi-isometry can be
made continuous.
\end{teorema}

\begin{proof} Let $X,X'$ be hyperbolic approximations of $Z,Z'$,
let us assume (without loss of generality, see \ref{cor6}) that
they are defined with the same parameter $r$, and let $V,V'$ be
their sets of vertices. Consider also, $\lambda,A$ the constants
of the characterization shown in \ref{pq} of being PQ-symmetric.

\underline{Claim 1}. For every vertex $v\in V$ there is a vertex
$v'\in V'$ for which the ball $B(v')$ contains $f(B(v))$ and such
that $l(v')$ is maximal. This is consequence of $f$ being bounded.

Consider $|V|$ the set of splitting points in $V$. Note that for
any $v\in |V|$, $B(v)$ is a non-degenerated ball.

Let us define first the map $F$ restricted to $|V|$. For every
$v\in |V|$, let $F(v)=v'$ with $v'$ any point holding the
condition in Claim 1. Note that, for any other point $v''$ with
the same condition, clearly $B(v')\cap B(v'')\neq \emptyset$ and
hence $|v'v''|\leq 1$.

\underline{Claim 2}. There exist constants $\lambda\geq 1$ and
$C_0>0$ such that for any pair of vertices $v_1,v_2$ in $|V|$ with
$B(v_1)\subset B(v_2)$ (in particular, if they are joined by a
radial geodesic)
\[\frac{1}{\lambda} |v_1v_2|-C_0\leq |F(v_1)F(v_2)|\leq \lambda
|v_1v_2|+C_0.\]

Let $k_1=l(v_1)$ and $k_2=l(v_2)$, and let us suppose $k_2\geq
k_1$. As we saw in Lemma \ref{diam1}, since $v_i$ are splitting
points, $r^{k_i+1} \leq diam(B(v_i))\leq 4r^{k_i}$. Then,
\[\frac{r}{4}\cdot r^{(k_2-k_1)} \leq \frac{r\cdot
r^{k_2}}{4r^{k_1}} \leq \frac{diam (B(v_2))}{diam (B(v_1))}\leq
\frac{4r^{k_2}}{r\cdot r^{k_1}}=\frac{4}{r}\cdot r^{(k_2-k_1)}.\]

This, together with \ref{quasiprop}, yields:

\[A\Big(\frac{r}{4}\cdot r^{(k_2-k_1)}\Big)^{\lambda}\leq
\frac{diam (f(B(v_2))}{diam (f(B(v_1)))}\leq
\frac{1}{A}\Big(\frac{4}{r}\cdot
r^{(k_2-k_1)}\Big)^{\frac{1}{\lambda}}.\]

Now, if $k'_1=l(v'_1)$ and $k'_2=l(v'_2)$, $r^{k'_i+1}\leq diam
(f(B(x_i,\epsilon_i)))\leq 4r^{k'_i}$, and therefore,
\[\frac{r}{4}\cdot r^{(k'_2-k'_1)}\leq
\frac{diam (f(B(v_2))}{diam (f(B(v_1)))}\leq \frac{4}{r}\cdot
r^{(k'_2-k'_1)}.\]

and hence we have:

\[A\Big(\frac{r}{4}\cdot r^{(k_2-k_1)}\Big)^{\lambda}\leq
\frac{4}{r}\cdot r^{(k'_2-k'_1)}\]

and

\[\frac{r}{4}\cdot r^{(k'_2-k'_1)}\leq \frac{1}{A}\Big(\frac{4}{r}\cdot r^{(k_2-k_1)}\Big)^{\frac{1}{\lambda}}.\]

obtaining that

\[\frac{r}{4}A\Big(\frac{r}{4}\cdot
r^{(k_2-k_1)}\Big)^\lambda\leq r^{(k'_2-k'_1)}\leq
\frac{4}{r}\frac{1}{A}\Big(\frac{4}{r}\cdot
r^{(k_2-k_1)}\Big)^{\frac{1}{\lambda}}.\]

Taking $C_1=log_r(A(\frac{r}{4})^{1+\lambda})$, $C_2=log_r
(\frac{1}{A}(\frac{4}{r})^{1+\frac{1}{\lambda}})$ and
$C_3:=\max\{|C_1|,|C_2|\}$, since $log_r$ is decreasing, it is
readily seen that

\[\frac{1}{\lambda}(k_2-k_1)-C_3\leq k'_2-k'_1 \leq \lambda (k_2-k_1)+C_3.\]

Since $B(F(v_1))\cap B(F(v_2))\neq \emptyset$, by Corollary 6.2.7,
$k'_2-k'_1\leq |F(v_1)F(v_2)|\leq k'_2-k'_1+1$ and making
$C_0:=C_3+1$, this proves Claim 2.

\underline{Claim 3}. There exist constants $\lambda\geq 0, \ C>0$
such that $F|_{|V|}$ is a $(\lambda,C)$ quasi-isometric map.
Consider any pair of vertices $v_1,v_2\in |V|$ such that none of
the balls $B(v_i)$ contains the other (hence, they are not joined
by a radial geodesic). Let $w$ be a branch point for them and
notice that by Lemma \ref{split}, $w\in |V|$. Now $|v_1w|+|wv_2|-1
\leq |v_1v_2| \leq |v_1w|+|wv_2|$. Let $w'$ be a branch point for
$F(v_1),F(v_2)$. Again, $|F(v_1)w'|+|w'F(v_2)|-1 \leq
|F(v_1)F(v_2)| \leq |F(v_1)w'|+|w'F(v_2)|$. Therefore, if there is
a constant $C_4$ depending only on $r,\lambda,A$ such that
$|F(w)w'|\leq C_4$, then Claim 4 follows immediately from Claim 3
substituting constant $C_0$ by $C=C_0+2C_4$. Let us show how the
existence of $C_4$ comes from (\ref{quasiprop}).

Let $l(w)=k_w$, and consider $B_2=B(v_1)\cup B(v_2)$ and
$B_1=B(w)$. As we saw in Lemma \ref{diam2}
\[\frac{r^2}{4} \leq \frac{diam (B_2)}{diam (B_1)}.\]

Applying \ref{quasiprop}, we obtain that

\begin{equation}\label{eq41} A\Big(\frac{r^2}{4}\Big)^\lambda \leq \frac{diam (f(B_2))}{diam (f(B_1))}
\leq 1.
\end{equation}

Clearly, $B(w')\cap B(F(w))\neq \emptyset$. Then $|w'F(w)|\leq
|k'_{F(w)}-k'_{w'}| +1$ where $k'_{F(w)}=l(F(w))$ and
$k'_{w'}=l(w')$, and it suffices to bound $|k'_{F(w)}-k'_{w'}|$.

It is immediate to check that
\[r^{k'_{w'}+1}\leq diam (B(F(v_1))\cup B(F(v_2)))\leq 4r^{k'_{w'}},\]
and
\[r^{k'_{F(w)}+1}\leq diam (f(B_1))\leq 4r^{k'_{F(w)}}.\]

We can apply Lemma \ref{union} with $A_i=f(B(v_i))$ and
$D_i=B(F(v_i))$. Since $diam (B(F(v_i)))\leq \frac{4}{r}diam
(f(B(v_i)))$, we obtain that
\begin{equation}\label{eq43} diam (B(F(v_1))\cup B(F(v_2))) \leq
\Big(\frac{16}{r}+2\Big)\cdot diam(f(B_2))
\end{equation}

From these, it can be readily seen that
\[A\Big(\frac{r^2}{4}\Big)^\lambda \leq \frac{diam (f(B_2))}{diam (f(B_1))}
\leq \frac{4}{r} r^{k'_{w'}-k'_{F(w)}}\] and
\[\frac{r^2}{4(16+2r)} r^{k'_{w'}-k'_{F(w)}} \leq \frac{diam (f(B_2))}{diam (f(B_1))}
\leq 1.\]

Taking logarithms to the base $r$ in these inequalities we can
bound $|k'_{F(w)}-k'_{w'}|+1$ with a constant $C_4$ only depending
on $A,\lambda,r$ and prove the claim.

\underline{Claim 4}. For any vertex $v\in V\backslash |V|$ such
that $B(v)$ is not degenerated nor the whole space, there exist
vertices $v_1,v_2\in |V|$ with $B(v_2)\subset B(v_1)$ and a radial
geodesic $[v_2,v_1]$ containing $v$. Moreover, we can ask
$v_1,v_2$ to be at maximal and minimal level respectively with
that property. This is immediate from the definition of splitting
point.

Consider $v\in V\backslash |V|$ and $v_1,v_2\in |V|$ as in Claim
4. Note that, though for $v'_1=F(v_1)$ and $v'_2=F(v_2)$
$l(v'_1)\leq l(v'_2)$, they might not be joined by a radial
geodesic. Nevertheless, $B(v'_1)\cap B(v'_2)\neq \emptyset$ and
therefore there is another vertex $w'_1$ which is in the same
level of $v'_1$, $l(w'_1)=l(v'_1)$, and such that $w'_1$ is joined
by a radial geodesic with $v'_2$ and by a horizontal edge with
$v'_1$. Now suppose that in the geodesic path $[v_1,v_2]$, abusing
of the notation, $v=t\cdot v_1+(1-t)\cdot v_2$ for some $t\in
(0,1)$. Let us define $F(v)=t\cdot w'_1+(1-t)\cdot v'_2$ in
$[w'_1,v'_2]$.

\underline{Claim 5}. A different choice of the points
$v_1,v_2,w'_1$ described in Claim 4 yields a different map, $F'$,
whose distance to $F$ is uniformly bounded by a constant. First,
note that any different choice for $v_1$ would be a vertex $u_1$
with $l(u_1)=l(v_1)$, $B(u_1)$ would intersect $B(v_1)$ and hence
$|u_2v_2|\leq 1$. From the definition of splitting point, it is
clear that $B(v_2)=B(v)$ and therefore, the election of $v_2$
is unique.
As we saw in Claim 3, $F|_{|V|}$ is a
$(\lambda,C)$-quasi-isometric map. Hence, 
$|w'_1w''_1|\leq \lambda +C+2$ for any vertex $w''_1$ at distance
1 from $F(u_1)$. It is clear that $|l(F(v))-l(F'(v))| \leq \lambda
+C+2$, and by Lemma \ref{intersect}, it is readily seen that
$|F(v)F'(v)|\leq 2\lambda +C+3$ concluding Claim 5.

Consider $v\in V$ such that $B(v)=\{v\}$. Let $k_v$ the minimal
$k$ such that $\{v\}\in V_k$. Obviously, for every $k\geq k_v$,
$\{v\}\in V_k$ so let us denote $v_k$ the vertex at level $k$
given by this degenerated ball ($B(v_k)=\{v\}$). There is some
vertex $u\in V_{k_v-1}$ such that $\{v\}\in B(u)$. Consider
$k'=l(F(u))$ and define $k_{f(v)}=k'+1$. Now, for every $k\geq
k_{f(v)}$ let $v'_k$ be a vertex in $V'_k$ containing $\{f(v)\}$
and joined by a radial edge with $v'_{k+1}$ (this can be always
considered taking vertices in $B(f(v),r^k)\cap V'_k$). Finally,
let us define $F(v_k)=v'_{k'}$ with $k'=k-k_v+k_{f(v)}$ for every
$k\geq k_v$.

Note that for $k > k_v$, the vertex $v_k$ is only joined by radial
edges with $v_{k-1},v_{k+1}$ and that $|F(v_{k_v})F(u)|\leq 2$.
Thus, the radial geodesic ray from $v_{k_v}$ towards $\{v\}$ is
sent isometrically to a radial geodesic ray from $v'_{k_{f(v)}}$
towards $\{f(v)\}$.

If the metric spaces $Z,Z'$ are unbounded then $F$ is already
defined on $V$. The only case left, is when $v$ is such that
$B(v)=Z$, but then, considering the truncated hyperbolic
approximations, it suffices to make $F(v)$ the vertex in the
minimal level of $X'$.

\underline{Claim 6}. There exist constants $\lambda \geq 1,\ C'>0$
such that $F|_V$ is a $(\lambda,C')$ quasi-isometric map on the
vertices. We already proved this for vertices in $|V|$. Let
$v_1,v_2 \in V$ representing non-degenerated balls. If $v_1,v_2$
are joined by a radial geodesic the claim is immediate from the
construction. Otherwise, let $w$ be a branch point for them (then
$w\in |V|$). Thus, the upper bound will be clear from Claim 2 and
the construction of $F$. The same argument from Claim 3 on the
existence of $C_4$ gives us now the lower bound.

If $v_i$ for $i=0$ or $1$ is a degenerated ball, consider the
minimal level $k_i$ such that $\{v_i\}\in V_{k_i}$ and its ball is
still degenerated and let $u_i$ be a vertex with $l(u_i)=k_{i-1}$
and such that $B(v_i) \subsetneq B(u_i)$. If $v_i$ has at least
two points, just let $u_i=v_i$. $F$ is a
$(\lambda,C)$-quasi-isometry on $u_1,u_2$ and
$|F(u_i)F(v_i)|=|u_iv_i|$ where the geodesics (in case $u_i\neq
v_i$) $[u_i,v_i],[F(u_i),F(v_i)]$ are radial and isometric.

If $F(u_1),F(u_2)$ are distinct and not joined by a radial
geodesic then
\[ |F(v_1)F(v_2)|= |F(v_1)F(u_1)|+|F(u_1)F(u_2)|+|F(u_2)F(v_2)|\]
and it follows that $F$ is a $(\lambda,C)$-quasi-isometry on
$v_1,v_2$.

Otherwise, the upper bound, $|F(v_1)F(v_2)|\leq \lambda
|v_1,v_2|+C$, is clear but not the lower one.

Note that for any branch point $b(F(v_1),F(v_2))$ of
$F(v_1),F(v_2)$, \[|F(v_1)
b(F(v_1),F(v_2))|+|b(F(v_1),F(v_2))F(v_2)|-1\leq
|F(v_1)F(v_2)|\leq \] \[ \leq |F(v_1)
b(F(v_1),F(v_2))|+|b(F(v_1),F(v_2))F(v_2)|.\]

Then, it suffices to check that the distance between
$b(F(v_1),F(v_2))$ and $b(F(u_1),F(u_2))$ is bounded by a constant
$C_5=C_5(A,\lambda,r)$ to assure that \[|F(v_1)F(v_2)|\geq
|F(v_1)F(u_1)|+|F(u_1)F(u_2)|+|F(u_2)F(v_2)|-2C_5-1 \geq\]
\[ \geq |v_1u_1|+\frac{1}{\lambda}|u_1u_2| -C +|u_2v_2|-2C_5-1.\] This
implies that making $C':=C+2C_5+1$, $F$ is a $(\lambda,C')$
quasi-isometric map.

Clearly, by definition of $F(v_i)$, any branch point,
$b(F(u_1),F(u_2))$, of $F(u_1),F(u_2)$ is a cone point for
$F(v_1),F(v_2)$. Then, it suffices to bound
$l(b(F(v_1),F(v_2)))-l(b(F(u_1),F(u_2)))\leq C_6$ for some
constant $C_6=C_6(\lambda,A,r)$ and make $C_5=C_6+1$. Let us see
how this comes from \ref{quasiprop}.

Let $B_2=B(v_1)\cup B(v_2)$ and $B_1=B(b(u_1,u_2))$ with
$b(u_1,u_2)$ a branch point for $u_1,u_2$. Since $b(v_1,v_2)$
intersects both $B(u_1)$ and $B(u_2)$, it can be readily seen that
$|b(v_1,v_2)b(u_1,u_2)|\leq 1$ and therefore it is immediate to
check that \[\frac{r^2}{4} diam (B(b(v_1,v_2)))\leq diam
(B(b(u_1,u_2)))\leq \frac{4}{r^2} diam (B(b(v_1,v_2))).\] Also, by
lemma \ref{diam2},
\[diam(B(v_1)\cup B(v_2))\geq \frac{r^2}{4} diam (B(b(v_1,v_2))).\] Thus,
\[\frac{diam (B_2)}{diam (B_1)}\geq \frac{r^2 diam(B(v_1)\cup B(v_2))}{4 \cdot diam
(B(b(v_1,v_2)))}\geq \frac{r^4}{4^2},\] and applying
\ref{quasiprop},
\[\frac{diam (f(B_2))}{diam (f(B_1))}\geq
A\Big(\frac{r^4}{4^2}\Big)^\lambda.\]

Note that $F(b(u_1,u_2))$ is a vertex in a maximal level with
$f(B_1) \subset B(F(b(u_1,u_2)))$ and therefore, $diam (f(B_1))
\geq \frac{r}{4}diam (B(F(b(u_1,u_2))))$. Then it follows that
\[\frac{diam(B(b(F(v_1)), B(F(v_2))))}{diam (B(F(b(u_1,u_2))))}
\geq \frac{diam(f(B_2))}{diam (B(F(b(u_1,u_2))))}\geq
\frac{r}{4}A\Big(\frac{r^4}{4^2}\Big)^\lambda.\]

From this, together with the existence of a constant
$C_4=C_4(A,\lambda,r)$ such that
$|F(b(u_1,u_2))b(F(u_1),F(u_2))|\leq C_4$ which was proved in
Claim 3, is easily obtained $C_6$ proving the claim in the
unbounded case. If $Z,Z'$ are bounded it is trivial to check that
the claim is true also when we consider the image of $v$ for
$B(v)=Z$.

\underline{Claim 7}. $F|_V$ can be extended to a quasi-isometric
map on the hyperbolic approximation $X$. This is immediate since
we have already defined this map in a set which is 1-dense (i.e.
any point in the space is at distance $\leq 1$ to the set) in the
hyperbolic approximation. Also notice that if for every edge
$[v,v']$ in $X$ we define the image to be any geodesic path
$[F(v),F(v')]$ in $X$, $F$ is in fact a continuous quasi-isometric
map.

\underline{Claim 8}. $F$ is a quasi-isometry. Let us see that
there is a constant $C_7$ such that for any $v'\in V'$,
$|v'F(X)|\leq C_7$. Since $f$ is a homeomorphism there is some
ball and, in particular, some vertex $v_2\in V$ such that
$f(B(v_2))\subset B(v')$. Also, by \ref{mp}, the map is metrically
proper. In particular, there is a vertex $v_1\in V$ such that
there is a radial geodesic from $v_2$ to $v_1$ and $B(v')\subset
f(B(v_1))$ with $l(F(v_1)) \geq l(v')$.

The map $F$ is a $(\lambda,C)$ quasi-isometry on $[v_2,v_1]$.
Then, since $l(F(v_2))\geq l(v')\geq l(F(v_1))$, there is a vertex
$v\in [v_2,v_1]$ such that $|l(F(v))-l(v')| \leq \lambda + C$.
Clearly, $B(F(v))$ intersects $B(F(v_2))$ and, therefore, $B(v')$.
Hence, by Lemma \ref{intersect}, $d(v',F(X))\leq \lambda +C
+1=:C_7.$

\underline{Claim 9}. The induced map in the boundary is $f$, i.e.
$\partial_\infty F(z)=f(z) \, \forall z\in Z$. Any point $z\in Z$
can be identified with a point in $\partial_\infty X$ given by a
sequence of vertices $\{v_k\}_{k\in \mathbb{N}}$ such that $v_k\in
V_k$ and $z\in B(v_k)$, which clearly converges at infinity. The
sequence $F(v_k)$ also converges at infinity and hence, defines a
point $\partial_\infty F(z)$ in $Z'$, and it is necessarily
$f(z)$, which is, by construction, in $B(F(v_k))$ for every $k$.
\end{proof}

Let us recall Corollary 7.1.6 in \cite{BS}, in which visual
metrics with base points in the spaces are supposed on the
boundaries:

\begin{teorema}[Buyalo-Schroeder] \label{cor6} Visual hyperbolic geodesic spaces $X,X'$ with
bilipschitz equivalent boundaries at infinity are roughly similar
to each other. In particular, every visual hyperbolic space is
roughly similar to any hyperbolic approximation of its boundary at
infinity; and any two hyperbolic approximations of a complete
bounded metric space $Z$ are roughly similar to each other.
\end{teorema}

From \ref{tma} and \ref{cor6}, we have:

\begin{cor} Let $X,X'$ be visual hyperbolic geodesic spaces. Then,
any PQ-symmetric homeomorphism $f:\partial_\infty X \to
\partial_\infty X'$ (with
respect to any visual metrics between their boundaries with base
points in $X$, $Y$ respectively) can be extended to a
quasi-isometry $F:X \to X'$.
\end{cor}

The next result also appears in \cite{BS}:

\begin{teorema} Let $f:X\to Y$ be a quasi-isometric map of
hyperbolic spaces. Then $f$ naturally induces a well-defined map
$\partial_\infty f:\partial_\infty X \to \partial_\infty Y$ of
their boundaries at infinity which is PQ-symmetric with respect to
any visual metrics with base points in $X$, $Y$ or with base
points $\omega\in \partial_\infty X$, $\partial_\infty f(\omega)
\in
\partial_\infty Y$ respectively.
\end{teorema}

From this and \ref{tma}, we conclude that,

\begin{cor} Two visual hyperbolic geodesic spaces $X,\ Y$ are quasi-isometric
if and only if there is a PQ-symmetric homeomorphism $f$ with
respect to any visual metrics between their boundaries with base
points in $X$, $Y$ respectively.
\end{cor}

\end{document}